\pdfoutput=1
\documentclass[11pt]{amsart}

\usepackage{amsmath,amssymb,amsthm}
\usepackage[margin=1.2in]{geometry}
\usepackage{booktabs}
\usepackage{adjustbox}
\usepackage{xcolor}
\usepackage[colorlinks=true,linkcolor=blue,citecolor=blue,urlcolor=blue]{hyperref}
\hypersetup{pdfauthor={Bernd Johannes Wuebben},
  pdftitle={Non-smoothable Calabi-Yau threefolds from reflexive polytopes}}

\newtheorem{theorem}{Theorem}[section]
\newtheorem{proposition}[theorem]{Proposition}
\newtheorem{lemma}[theorem]{Lemma}
\newtheorem{corollary}[theorem]{Corollary}
\theoremstyle{definition}

\newtheorem{example}[theorem]{Example}
\newtheorem{question}[theorem]{Question}
\theoremstyle{remark}
\newtheorem{remark}[theorem]{Remark}

\newcommand{\ZZ}{\mathbb{Z}}
\newcommand{\RR}{\mathbb{R}}
\newcommand{\CC}{\mathbb{C}}
\newcommand{\PP}{\mathbb{P}}
\newcommand{\Def}{\operatorname{Def}}
\newcommand{\Sing}{\operatorname{Sing}}
\newcommand{\Spec}{\operatorname{Spec}}
\newcommand{\conv}{\operatorname{conv}}
\newcommand{\FF}{\mathbb{F}}

\title{Non-smoothable Calabi--Yau threefolds \\from reflexive polytopes}
\author{Bernd Johannes Wuebben}
\date{September 9, 2026}

\begin{document}

\begin{abstract}
We construct projective Calabi--Yau threefolds with isolated Gorenstein
canonical singularities that admit no smoothing, as anticanonical
hypersurfaces in Gorenstein Fano toric fourfolds. One example has a unique
singularity, the anticanonical cone over the first Hirzebruch surface.
Another has a unique singularity whose reduced miniversal base is a smooth
curve but whose deformations are all singular. The construction combines
Altmann's deformation theory with the smoothing criterion of
Corti--Filip--Petracci: a two-dimensional face of the ambient fan polytope with primitive edges
obstructs smoothing if it has no Minkowski decomposition into unit segments
and unimodular triangles. We also classify the 217 isolated Gorenstein
toric threefold singularities whose polygon edge vectors have coordinates
in $[-2,2]$ in some lattice basis, recording their deformation types and
smoothing components. Subject to completeness and absence of repetitions
in the database copy used, a scan of the Kreuzer--Skarke classification
finds this obstruction in $39{,}175{,}536$ of its $473{,}800{,}776$
polytopes (approximately $8.27\%$), whose generic anticanonical
hypersurfaces therefore admit no smoothing.
\end{abstract}

\maketitle
\enlargethispage{5pt}

\section{Introduction}\label{sec:intro}

Let $X$ be a projective threefold with trivial dualizing sheaf,
$h^1(\mathcal{O}_X)=0$, and Gorenstein canonical singularities; we call such
an $X$ a \emph{Calabi--Yau threefold} with canonical singularities.  A basic question,
going back to the foundational work of Friedman \cite{Friedman86}, is whether
$X$ can be \emph{smoothed}: does there exist a flat proper family over a
pointed disc with special fibre $X$ and smooth general fibre?

For the mildest singularities the question is completely understood.  If $X$
has ordinary double points, Friedman \cite{Friedman86} showed that a smoothing
exists if and only if the exceptional curves of a small resolution satisfy a
relation in homology with every coefficient nonzero; if $X$ is
$\mathbb{Q}$-factorial and has isolated terminal Gorenstein singularities,
Namikawa and Steenbrink \cite[Theorem~1.3]{NamikawaSteenbrink95} proved that
$X$ is smoothable.  Beyond the
terminal case the picture changes qualitatively.  Gross \cite{Gross97}
analyzed the case where $X$ arises from a primitive type~II contraction of a
smooth Calabi--Yau threefold (a divisor $E$ contracted to a point) and proved
that $X$ is then smoothable \emph{unless} $E \cong \PP^2$ or $E \cong \FF_1$
\cite[Theorem~5.8]{Gross97}.  The two exceptions are
non-smoothable: the germ at the singular point is the anticanonical cone over
$E$, which for $E=\PP^2$ is the rigid quotient singularity
$\frac13(1,1,1)$ \cite{Schlessinger71}, and for $E = \FF_1$ has
one-dimensional tangent space $T^1$ but no nontrivial deformations over a
reduced base
\cite[Example~2.19]{Gross97}.
Deformations of Calabi--Yau threefolds with canonical and log canonical
singularities remain an active subject; see Friedman--Laza
\cite{FriedmanLaza25,FriedmanLazaSigma} for recent progress.

We give explicit reflexive polytopes whose anticanonical hypersurfaces have
one non-smoothable singular point, including a germ with a positive-dimensional
reduced deformation base.  The inputs are:

\begin{enumerate}
\item Altmann's theorem \cite{Altmann97} that for an isolated Gorenstein toric
threefold singularity (the cone $C(Q)$ over a lattice polygon $Q$ with
primitive edges) the irreducible components of the reduced miniversal base
correspond to the maximal decompositions of $Q$ into Minkowski sums of lattice
polytopes;
\item the smoothing criterion of Corti--Filip--Petracci
\cite[\S5.1]{CFP}: the \emph{smoothing} components correspond to the Minkowski
decompositions of $Q$ into \emph{unit segments} and \emph{standard triangles};
consequently $C(Q)$ is smoothable if and only if such a decomposition exists;
\item Batyrev's construction \cite{Batyrev94}: for a reflexive polytope
$\Delta \subset \ZZ^4$ (throughout, the \emph{fan} polytope; see
\S\ref{sec:planting}), the generic anticanonical hypersurface $X$ in the
Gorenstein Fano toric fourfold $\PP_\Delta$ is a Calabi--Yau threefold with
canonical singularities, and (as we spell out in
Proposition~\ref{prop:planting}) when all edges of $\Delta$ have lattice
length one the singular points of $X$ have germs $C(F)$ for the
two-dimensional faces $F$ of $\Delta$ whose cones are singular; each such
face contributes as many points as the lattice length of its dual edge.
\end{enumerate}

Putting these together: \emph{if a reflexive $4$-polytope $\Delta$ has a
two-dimensional face $F$ with unit edges whose associated cone $C(F)$ is
non-smoothable, then the generic anticanonical hypersurface
$X \subset \PP_\Delta$ is a Calabi--Yau threefold that admits no smoothing}
(Corollary~\ref{cor:global}; by Remark~\ref{rem:longedges} only the edges
of $F$ itself need lattice length one).  The point
is that non-smoothable faces can be recognized by a finite test on their edge
vectors (Theorem~\ref{thm:trichotomy}) and realized by explicit completions
of embedded polygons (\S\ref{sec:examples}).

Our principal examples have a single singular point.  The first realizes
the anticanonical cone over $\FF_1$; the second realizes a non-smoothable
germ with nontrivial deformations over a reduced curve.

\begin{theorem}[Theorem~\ref{thm:C}]\label{thm:introA}
There is a $7$-vertex reflexive $4$-polytope $\Delta_C$ whose generic
anticanonical hypersurface $X_C\subset\PP_{\Delta_C}$ is a Calabi--Yau
threefold with a unique singular point, analytically isomorphic to the
anticanonical cone over the Hirzebruch surface $\FF_1$.  Every deformation of
$X_C$ over a reduced base preserves this singular germ; in particular
$X_C$ admits no smoothing.  Its maximal projective crepant partial resolution
is smooth with $(h^{1,1},h^{2,1})=(4,108)$.
\end{theorem}

The local model here is precisely Gross's exceptional germ: the cone over
$\FF_1$ is \emph{reduced-rigid}.  Its miniversal base is a fat
point, with $\dim T^1 = 1$ but no positive-dimensional deformations
\cite[Example~2.19, \S5.5]{Gross97}.  Since it is the cone over a non-simplicial polygon, it
is not a quotient singularity, so $X_C$ is not covered by the classical
$\frac13(1,1,1)$ examples.  Compact Calabi--Yau threefolds containing such
points appear in the primitive-contraction literature: Gross constructs one
with two of them, as a double cover of the projective cone over the
anticanonically embedded $\FF_1$, branched over a general quartic section
\cite[Example~2.19]{Gross97}, while in \cite{KapustkaKapustka09,Kapustka09}
the smoothable exceptional types are systematically smoothed and the $\FF_1$
case is excluded.  There is also a single-point construction implicit in
\cite[Example~4.1]{Gross97}: specialize the elliptic fibration there to
$S=\FF_1$ and contract its section.  Adjunction gives the section normal
bundle $K_{\FF_1}$, and the analytic-neighborhood argument in the proof of
\cite[Proposition~5.4]{Gross97} identifies the contraction germ with the
anticanonical cone.  This is a specialization of Gross's construction,
rather than an example separately stated there.  Theorem~\ref{thm:C}
provides an explicit realization by a reflexive polytope.

For the second example let
\[
Q_B=\conv\{(0,0),(-2,-1),(-3,-2),(-2,-3),(-1,-2)\},
\]
and write $C(Q_B)$ for the affine toric threefold defined by the cone over
$Q_B$ at height one.  This pentagon is the Minkowski sum of the triangle
defining the quotient singularity $\frac13(1,1,1)$ and a unit segment.

\begin{theorem}[Theorem~\ref{thm:D}]\label{thm:introB}
There is a $10$-vertex reflexive $4$-polytope $\Delta_D$ whose generic
anticanonical hypersurface $X_D\subset\PP_{\Delta_D}$ is a Calabi--Yau
threefold with a unique singular point of analytic type $C(Q_B)$.
The reduced miniversal base of this germ is a smooth curve, and its
general deformation has one $\frac13(1,1,1)$ singular point.  The germ has
no smoothing component, so $X_D$ admits no smoothing.  Its maximal
projective crepant partial resolution is smooth with
$(h^{1,1},h^{2,1})=(8,118)$.
\end{theorem}

The local family in the second theorem is Altmann's family for the
triangle-plus-segment decomposition.  Its general fibre retains the quotient
cone, which is rigid (Lemma~\ref{lem:cascade}).  The reduced deformation
curve does not imply infinitesimal unobstructedness: this germ has
$\dim T^1=2$ and a nonreduced miniversal base.  Gross already constructed
non-smoothable compact Calabi--Yau threefolds with locally smoothable
singularities by contracting sections of elliptic fibrations over del Pezzo
surfaces \cite[Example~4.1]{Gross97}.  Here non-smoothability is intrinsic to
the singular germ, even though the germ admits nontrivial reduced
deformations; no obstruction to globalizing local smoothings is needed.

The proofs begin with simpler three-point examples $X_A$ and $X_B$
(Theorems~\ref{thm:A} and~\ref{thm:B}).  Each is obtained by adjoining two
vertices to a polygon embedded at height one.  Their prescribed faces have
dual edges of length $3$, hence contribute three singular points.
More generally, this two-vertex construction forces dual-edge length at
least $2$ (Lemma~\ref{lem:twovertices}).  A third adjoined vertex allows
length $1$ and gives $X_C$, as well as a threefold with a single
$\frac13(1,1,1)$ point.  A systematic scan of the Kreuzer--Skarke
classification \cite{KreuzerSkarke00} locates a $10$-vertex reflexive
polytope giving $X_D$; its printed vertex and facet data verify the example
independently of the database-completeness hypothesis used for the scan's
aggregate counts.

The two theorems are instances of a general construction: realize a
prescribed polygon as a two-dimensional face of a reflexive $4$-polytope.
We state this as Proposition~\ref{prop:planting} and
Corollary~\ref{cor:global}; it
gives a sufficient local obstruction to smoothing in terms of polygon
combinatorics.  To describe its local range
we record in \S\ref{sec:local} a complete census of the isolated Gorenstein
toric threefold singularities whose polygon has an edge representation with
coordinates in $[-2,2]$: there are $217$ classes up to equivalence,
comprising the ordinary double point, $8$ reduced-rigid classes,
$79$ deformable-but-non-smoothable
classes, and $129$ smoothable classes.  The five classes whose polygon has one
interior lattice point are exactly the anticanonical cones over the five
smooth toric del Pezzo surfaces, and the census reproduces the classical
deformation-theoretic data of these cones \cite[\S5.5]{Gross97} exactly
(Remark~\ref{rem:delpezzo}); this both validates the computation and places
the new examples in context: non-smoothability is common among Gorenstein
toric threefold singularities (87 of the 216 non-terminal canonical classes
in the census),
while the del Pezzo cone catalogue sees only its smallest stratum.

Our approach is the Calabi--Yau analogue, one dimension up, of the program of
Petracci and Corti--Hacking--Petracci on (non-)smoothability of Gorenstein
toric \emph{Fano} threefolds \cite{Petracci20,CHP24}, where the local models
are the same cones $C(F)$, attached to the facets of a reflexive
$3$-polytope.  In our setting the local-to-global step is simpler than in
loc.~cit.: because the relevant local models are non-smoothable outright, no
almost-flatness or bundle-cohomology obstruction is needed.  The versal base
of the germ already has no smoothing component, and this kills all global
smoothings (Lemma~\ref{lem:localglobal}).

All computations in this paper (the census, the face inventories of the
example polytopes, the dual-edge point counts, the completion-search
statistics and
the Kreuzer--Skarke scan of \S\ref{sec:questions}, and the Hodge numbers) are
implemented in six short exact-arithmetic computer programs, validated
against the classical cases recalled in the text (the quintic; $X_9 \subset
\PP(1,1,1,3,3)$; the del Pezzo cone catalogue; Altmann's hexagon; and, for
the scan, the Hodge numbers recorded in the Kreuzer--Skarke data itself).
They are included with the arXiv submission as ancillary files and are
public, together with all data and results, at
\url{https://github.com/bwuebben/calabi-yau-smoothability}.

\subsection*{Acknowledgements}
I am grateful to Mark Gross for first drawing my attention, many years ago, to the problem of smoothing canonical singularities on Calabi--Yau threefolds.
The author used LLMs during the development of this work for literature exploration, editorial and coding assistance and as an adversarial reader.

\section{Local models: cones over polygons}\label{sec:local}

\subsection{The dictionary}
Let $N \cong \ZZ^3$ and let $Q \subset N_\RR$ be a lattice polygon placed at
height one, i.e.\ $Q \subset \{ \langle u^*, \cdot\rangle = 1\}$ for a
primitive $u^* \in M = N^\vee$.  Write $\sigma_Q$ for the cone over $Q$ and
\[
C(Q) \;=\; \Spec \CC[\sigma_Q^\vee \cap M].
\]
Every Gorenstein toric affine threefold without torus factors arises this way,
and $C(Q)$ depends only on $Q$ up to affine unimodular equivalence of the
plane $\{ \langle u^*,\cdot\rangle = 1\}$, i.e.\ up to $\ZZ^2 \rtimes
GL_2(\ZZ)$ acting on $Q$.  We freely identify $Q$ with a polygon in $\ZZ^2$.
The basic dictionary (see e.g.\ \cite{CFP}) is:
\begin{itemize}
\item $C(Q)$ has an \emph{isolated} singularity if and only if every edge of
$Q$ has lattice length one (\emph{unit edges}); an edge of lattice length
$m\ge 2$ produces a curve of transverse $A_{m-1}$-singularities.
\item $C(Q)$ is canonical and Gorenstein; it is \emph{terminal} if and only if
$Q$ has no lattice points other than its vertices, and it is smooth if and
only if $Q$ is a standard triangle (unimodular image of
$\conv\{(0,0),(1,0),(0,1)\}$).
\item For a unit-edge $k$-gon, Pick's theorem gives $i = A - k/2 + 1$ interior
lattice points, where $A$ is the euclidean area.
\end{itemize}
Since translations do not change the edges, a unit-edge polygon is equivalent
data to its multiset $E(Q)=\{e_1,\dots,e_k\}$ of primitive edge vectors,
listed counterclockwise, with $\sum e_i = 0$; and $GL_2(\ZZ)$-equivalence of
polygons is $GL_2(\ZZ)$-equivalence of edge multisets.

\subsection{Deformations: the trichotomy}
Minkowski decompositions of unit-edge polygons are transparent at the level of
edge multisets.

\begin{lemma}\label{lem:partition}
Let $Q$ be a unit-edge polygon with edge multiset $E(Q)$.  Minkowski
decompositions $Q = R_1 + \dots + R_r$ into lattice polytopes (up to
translation of the summands) correspond bijectively to partitions
$E(Q)=E_1\sqcup\dots\sqcup E_r$ into non-empty zero-sum sub-multisets.  Under
this correspondence a summand is a unit segment if and only if its part is a
pair $\{v,-v\}$, and a standard triangle if and only if its part is a triple
$\{a,b,c\}$ with $a+b+c=0$ and $|\det(a,b)|=1$.
\end{lemma}

\begin{proof}
If $Q = R + S$, then each edge of $Q$ decomposes as the sum of the faces of
$R$ and $S$ with the same outer normal, and lattice lengths add; since the
edges of $Q$ are primitive, each edge comes entirely from $R$ or from $S$.
Hence $E(Q) = E(R) \sqcup E(S)$, and inductively any decomposition partitions
$E(Q)$ into zero-sum parts.  Conversely, a zero-sum sub-multiset of primitive
vectors, sorted counterclockwise, closes up to a convex lattice polygon
(a segment when the part is $\{v,-v\}$), and the Minkowski sum of the polygons
associated to the parts of a partition has edge multiset $E(Q)$; since a
convex polygon is determined by its edge multiset up to translation, that sum
is $Q$.  The last statement is immediate: a triangle with primitive edges
$a,b,c=-a-b$ is unimodular if and only if $|\det(a,b)|=1$.
\end{proof}

The deformation theory of $C(Q)$ is completely described by two results.

\begin{theorem}[Altmann \cite{Altmann97}; Corti--Filip--Petracci
{\cite[\S5.1]{CFP}}]\label{thm:trichotomy}
Let $Q$ be a unit-edge lattice polygon, not a standard triangle, and let
$V = C(Q)$.
\begin{enumerate}
\item The irreducible components of the reduced miniversal base space of $V$
are in bijection with the maximal Minkowski decompositions of $Q$ into lattice
summands; each component is smooth, and the component associated to a
decomposition with $r$ summands has dimension $r-1$.
\item The smoothing components are exactly those associated to decompositions
all of whose summands are unit segments and standard triangles.
\end{enumerate}
Consequently, with the notation of Lemma~\textup{\ref{lem:partition}} (this
repackaging, and the labels below, are ours), exactly
one of the following holds:
\begin{itemize}
\item[\textup{(R)}] \textup{(reduced-rigid)} $E(Q)$ has no proper non-empty zero-sum
sub-multiset.  Then the reduced miniversal base is a point: $V$ has no
positive-dimensional deformations, and is not smoothable.
\item[\textup{(D)}] \textup{(deformable, non-smoothable)} $E(Q)$ has a proper
zero-sum sub-multiset, but admits no partition into pairs $\{v,-v\}$ and
unimodular triples.  Then $V$ deforms non-trivially but no deformation has
smooth general fibre.
\item[\textup{(S)}] \textup{(smoothable)} $E(Q)$ admits such a partition; the
number of such partitions equals the number of smoothing components of the
miniversal base.
\end{itemize}
\end{theorem}

Thus \emph{reduced-rigid} means that the reduced miniversal base is a point;
it does not assert $T^1=0$.  For example, the
$\FF_1$ cone below has a non-zero tangent space but a non-reduced,
point-supported miniversal base.

For the dimensions in (1) and the concentration of $T^1$ in the Gorenstein
degree see also \cite{Altmann95}; for unit-edge $k$-gons one has
$\dim T^1 = k-3$.  (In the count of smoothing components, unit segments
and standard triangles are Minkowski-indecomposable, so decompositions
into them are automatically \emph{maximal} and index components in
Altmann's parametrization.)  For a two-variable Laurent polynomial $f$, Filip
\cite{Filip25} constructs a formal deformation of the associated, possibly
non-isolated, Gorenstein toric pair and proves that the general fibre of that
family is smooth if and only if $f$ is $0$-mutable.  This gives non-isolated
smoothing families and evidence for the full smoothing-component
correspondence conjectured in \cite{CFP}; it does not classify all components
of every non-isolated germ.

\begin{example}\label{ex:pentagon}
Small examples of the three classes are as follows.
\begin{itemize}
\item The $\frac13(1,1,1)$ singularity $=\CC^3/\mu_3$ is the cone over the
triangle $T$ with $E(T)=\{(-2,-1),(1,-1),(1,2)\}$; no proper subset sums to
zero, so it is rigid, recovering Schlessinger \cite{Schlessinger71}.
\item The ordinary double point is the cone over the unit square,
$E=\{\pm(1,0),\pm(0,1)\}$: one partition into two segment pairs, hence
smoothable with a single smoothing component ($xy-zw=t$).
\item The pentagon $Q_B = T + s$, the Minkowski sum of $T$ with a unit segment
$s$, has $E(Q_B)=\{(-2,-1),(-1,-1),(1,-1),(1,1),(1,2)\}$.  Its only proper
zero-sum sub-multisets are the pair $\{(-1,-1),(1,1)\}$ and the complementary
triple $E(T)$, which is not unimodular ($\det = 3$).  So $Q_B$ has exactly one
non-trivial maximal decomposition, $T+s$: the reduced miniversal base of
$C(Q_B)$ is a smooth curve, and there is no smoothing component; $C(Q_B)$
is deformable but not smoothable.  (Here $\dim T^1 = 2$, so the miniversal
base is in addition non-reduced.)
\end{itemize}
\end{example}

The next lemma describes the general fibre of the homogeneous deformation
attached to a Minkowski decomposition; we use it in Remark~\ref{rem:stuck}
to identify what the local deformations of our second example do to the
germ.

\begin{lemma}[General fibre of the Altmann family]\label{lem:cascade}
Let $Q = R_0 + \dots + R_r$ be a Minkowski decomposition of a unit-edge
polygon into lattice summands, let $f_0,\dots,f_r$ be the standard
basis of $\ZZ^{r+1}$, and let $\widetilde V$ be the Gorenstein toric
$(3+r)$-fold associated to the \emph{Cayley cone}
$\widetilde\sigma=\operatorname{Cone}\bigl(\bigcup_{i} R_i\times\{f_i\}\bigr)
\subset \RR^2\oplus\RR^{r+1}$; write $u_0,\dots,u_r$ for the dual basis of
$f_0,\dots,f_r$, viewed as characters of $\widetilde V$, and let
\[
\pi=\bigl(\chi^{u_0}-\chi^{u_1},\ \dots,\ \chi^{u_0}-\chi^{u_r}\bigr)\colon
\widetilde V \longrightarrow \CC^r,
\qquad \pi^{-1}(0)\cong C(Q),
\]
be Altmann's homogeneous deformation attached to the decomposition
\cite{Altmann95} \textup{(}cf.\ \cite[Remark~6.2]{CFP}\textup{)}.  Then:
\begin{enumerate}
\item the general fibre of $\pi$ has exactly one singular point of analytic
type $C(R_i)$ for each two-dimensional non-smooth summand $R_i$, and no
other singular points;
\item for the pentagon $Q_B = T + s$ of
Example~\textup{\ref{ex:pentagon}}, the family $\pi$ dominates the
positive-dimensional component of $\Def C(Q_B)$, whose general fibre
therefore has a single singular point, of type $\frac13(1,1,1)$.
\end{enumerate}
\end{lemma}

\begin{proof}
(1)  For two-dimensional $R_i$ the
orbit closure of the face $\operatorname{Cone}(R_i\times\{f_i\})$ is an
$r$-dimensional affine toric variety whose character lattice is freely
generated by the restrictions of the $u_j$, $j \neq i$, while
$\chi^{u_i}$ vanishes on it; so for generic $t\in\CC^r$ the $r$ equations of
the fibre $\pi^{-1}(t)$ determine on it the single reduced point
$\chi^{u_0}=t_i$, $\chi^{u_j}=t_i-t_j$ (with the convention $t_0 = 0$,
which makes the formula uniform in $i$), lying in the open orbit, and the
implicit-function argument in the proof of
Proposition~\ref{prop:planting}(2) below identifies the germ of
$\pi^{-1}(t)$ there with $(C(R_i),0)$, on the chart cut out by the
characters $\chi^{u_j}$, $j \neq i$, which form a basis of the
saturated annihilator of $\operatorname{Cone}(R_i\times\{f_i\})$ and
exhibit it as $C(R_i)\times(\CC^\ast)^r$.  On the orbit closure of a face of
$\widetilde\sigma$ meeting two distinct summand heights, some equation of
$\pi^{-1}(t)$ restricts to $0=t_i$ or $t_i=t_j$, so the general fibre avoids
these strata.  The remaining strata are the open orbit and the orbits of
the proper faces $\tau$ of the cones
$\operatorname{Cone}(R_i\times\{f_i\})$ (together with, for a
\emph{segment} summand $R_i$, the two-dimensional cone
$\operatorname{Cone}(R_i\times\{f_i\})$ itself, which the same
computation below covers), and there the fibre is smooth.
Over the open orbit this holds because $u_0,\dots,u_r$ extend to a basis of
the character lattice.  Along the orbit of a proper face $\tau$ the ambient
$\widetilde V$ is itself smooth: $\tau$ is spanned either by a single ray
$(v,f_i)$ with $v$ a vertex of $R_i$, which is primitive, or by the rays
over an edge $[v,v']$ of $R_i$, and since the edge is unit the pair
$(v,f_i),\ (v'-v,0)$ extends to a lattice basis.  The functionals $u_j$,
$j \neq i$ (including $u_0$ when $i \neq 0$), vanish on $\tau$ and form part of a basis of
the saturated lattice $\tau^\perp \cap M$, so on the smooth chart
$U_\tau \cong \CC^{\dim\tau}\times(\CC^\ast)^{3+r-\dim\tau}$ the characters
$\chi^{u_j}$, $j\neq i$, may be taken as coordinates of
the torus factor, while $\chi^{u_i}$ is divisible by a nonconstant monomial
in the $\CC^{\dim\tau}$-coordinates and hence vanishes along the orbit
together with its derivatives in the torus coordinates.  The Jacobian of
the $r$ equations of $\pi^{-1}(t)$ with respect to
these torus coordinates is therefore triangular with diagonal entries
$\pm 1$, and the fibre is a
smooth complete intersection near the orbit.

(2)  By Example~\ref{ex:pentagon} the reduced miniversal base of $C(Q_B)$
is a smooth curve and $T+s$ is the unique non-trivial maximal
decomposition, with $r=1$.  By (1) the general fibre of $\pi$ has a single
singular point of type $C(T) = \frac13(1,1,1)$, not $C(Q_B)$; the family is
therefore non-trivial, and its classifying map to the
miniversal base is non-constant and dominates the curve.  Hence the
general fibre over the positive-dimensional component of $\Def C(Q_B)$ is
the general fibre of $\pi$: a single $\frac13(1,1,1)$ point.  The only
available deformation of the germ thus trades the pentagon cone for the
rigid quotient cone, after which no further deformation is possible.
\end{proof}

\paragraph{Maximal decompositions.}
In a maximal decomposition $Q=\sum_i R_i$, each summand is
Minkowski-indecomposable.  Consequently every two-dimensional non-smooth
summand is of type \textup{(R)}.  Combining
Theorem~\ref{thm:trichotomy} with Lemma~\ref{lem:cascade}, the general fibre
on the miniversal component indexed by this decomposition has precisely one
type-\textup{(R)} point $C(R_i)$ for each such summand and is otherwise
smooth.  Thus the general deformation on each positive-dimensional reduced
component of a type-\textup{(D)} cone has only reduced-rigid singularities.

\subsection{A census, and the del Pezzo catalogue}\label{subsec:census}
Both tests in Theorem~\ref{thm:trichotomy} are finite, so the trichotomy can
be tabulated.  We enumerated all unit-edge polygons whose (counterclockwise)
edge vectors have coordinates in $[-2,2]$ and reduced modulo $GL_2(\ZZ)$.
The possible numbers of edges are $3$ through $14$, and $16$.
Indeed, the box contains $16$ primitive vectors, whose sum is zero;
omitting one leaves a nonzero sum, so no $15$-gon occurs.

\begin{proposition}\label{prop:census}
Up to equivalence there are exactly $217$ unit-edge polygons whose
\textup(counterclockwise\textup) primitive edge vectors can be chosen with
all coordinates in $[-2,2]$, excluding the standard triangle.  They
comprise: the unit square (the ordinary double point, the unique terminal
class, smoothable); $8$ reduced-rigid classes; $79$ deformable-but-non-smoothable
classes; and $129$ smoothable classes.  The $8$ reduced-rigid classes are, listed as
$[k,i;\,E(Q)]$:
\[
\begin{array}{l}
[3,1;\,(-2,-1),(1,-1),(1,2)]\\[2pt]
[4,1;\,(-2,-1),(0,-1),(1,0),(1,2)]\\[2pt]
[4,2;\,(-2,-1),(2,-1),(1,2),(-1,0)]\\[2pt]
[4,2;\,(-2,-1),(2,-1),(1,1),(-1,1)]\\[2pt]
[5,2;\,(-2,-1),(2,-1),(1,1),(0,1),(-1,0)]\\[2pt]
[5,3;\,(-2,-1),(-1,-1),(2,-1),(1,2),(0,1)]\\[2pt]
[5,4;\,(-2,-1),(0,-1),(2,-1),(1,2),(-1,1)]\\[2pt]
[6,5;\,(-2,-1),(-1,-1),(2,-1),(1,0),(1,2),(-1,1)]
\end{array}
\]
\end{proposition}

\begin{proof}
Enumerate the subsets of the $16$ primitive vectors in $[-2,2]^2$.
Retain those that sum to zero and span $\RR^2$, and order each by angle.
These give exactly the polygons under consideration: convexity precludes
repeated primitive edge directions, and conversely each retained subset
closes to a convex unit-edge polygon.  The finite enumeration gives
$1{,}471$ edge multisets.  Testing their zero-sum partitions as in
Lemma~\ref{lem:partition} determines the deformation type and the number of
smoothing components.

For the reduction to equivalence classes, a finite bound is sufficient.  If
$g \in GL_2(\ZZ)$ carries one census edge multiset onto another, pick two
independent edges $u, v$ of the first, with images $u', v'$; then
$g = [\,u'\,|\,v'\,]\operatorname{adj}[\,u\,|\,v\,]/\det[\,u\,|\,v\,]$
with all eight entries of the two outer matrices in $[-2,2]$ and
$|\det[\,u\,|\,v\,]| \ge 1$, so every entry of $g$ is bounded by $8$.
Reduction under these matrices gives $218$ classes, including the standard
triangle.  Excluding it gives the $217$ classes in the statement; the
zero-sum partition test gives their stated types and the eight printed
reduced-rigid representatives.  The ancillary program supplies the candidate
enumeration, the polygon tests, and the bounded unimodular transformations
used in this calculation.\footnote{\texttt{toric\_census.py}.}
\end{proof}

The deformation types and numbers of smoothing components are
$GL_2(\ZZ)$-invariants of the edge multiset.  The results agree with the
classical cases: the conifold, the rigid $\frac13(1,1,1)$ cone,
Altmann's $\mathrm{dP}_6$ hexagon with its two smoothing components,
and the del Pezzo cones in Remark~\ref{rem:delpezzo}.
Reduced rigidity is not confined to polygons with few edges: for every
$k\ge3$ there are reduced-rigid unit-edge $k$-gons, for example with
$E(Q) = \{(1,a_1),\dots,(1,a_{k-1}),\, (-(k-1),-\textstyle\sum a_j)\}$ for
distinct $a_j$ with $\gcd(k-1,\sum a_j)=1$: any zero-sum subset must contain
the last vector and then all the others.

\begin{remark}[Anticanonical cones over toric del Pezzo surfaces]\label{rem:delpezzo}
The census contains exactly five classes with $i=1$ interior point.  If $Q$ is
unit-edge with $i=1$ interior lattice point, then $Q$ is reflexive (after
centering) and one checks
that $\sigma_Q^\vee$ is the cone at height one over the polar dual $Q^\circ$;
hence
\[
C(Q) \;\cong\; \Spec \textstyle\bigoplus_{m\ge 0} H^0\!\left(S,\,
\omega_S^{-m}\right),
\]
the anticanonical cone over the toric surface $S$ with moment polygon
$Q^\circ$.  For the five census classes the surfaces $S$ are smooth: they are
precisely the five smooth toric del Pezzo surfaces, and the census
reproduces the classical deformation theory of their cones
\cite[\S5.5]{Gross97} (which rests on \cite{Altmann97}):
\begin{center}
\adjustbox{max width=\textwidth}{%
\begin{tabular}{lcccc}
\toprule
$S$ & $\deg S$ & $[k,i]$ & $\dim T^1 = k-3$ & type \\
\midrule
$\PP^2$ & $9$ & $[3,1]$ & $0$ & rigid \\
$\FF_1$ & $8$ & $[4,1]$ & $1$ & reduced-rigid (fat point) \\
$\PP^1{\times}\PP^1$ & $8$ & $[4,1]$ & $1$ & smoothable, $1$ component \\
$\mathrm{dP}_7$ & $7$ & $[5,1]$ & $2$ & smoothable, $1$ component \\
$\mathrm{dP}_6$ & $6$ & $[6,1]$ & $3$ & smoothable, $2$ components \\
\bottomrule
\end{tabular}}
\end{center}
(The pair $[k,i]$ alone does not determine the class: $\FF_1$ and
$\PP^1{\times}\PP^1$ share $[4,1]$ and differ in type.)
In particular the second reduced-rigid class of Proposition~\ref{prop:census},
$Q_A$ with $E(Q_A)=\{(-2,-1),\allowbreak(0,-1),\allowbreak(1,0),\allowbreak(1,2)\}$, satisfies
$Q_A^\circ = $ the moment polygon of $(\FF_1, -K)$, so
\emph{$C(Q_A)$ is the anticanonical cone over $\FF_1$}, the exceptional
germ of \cite[Theorem~5.8]{Gross97}.  The deformable-but-non-smoothable
classes all have $i \ge 2$; they are cones over polarized toric surfaces that
are not anticanonically polarized del Pezzo surfaces, and so lie outside the
catalogue above.
\end{remark}

\section{Two-dimensional faces of reflexive polytopes and Batyrev
hypersurfaces}
\label{sec:planting}

Let $\widetilde N \cong \ZZ^4$ and let $\Delta \subset \widetilde N_\RR$ be a
reflexive polytope: a lattice polytope with $0$ in its interior whose polar
dual $\Delta^\circ=\{u : \langle u, v\rangle \ge -1 \ \forall v \in \Delta\}$
is again a lattice polytope.  Let $\PP_\Delta$ be the toric fourfold of the
face fan of $\Delta$; it is Gorenstein Fano, $-K_{\PP_\Delta}$ is ample and
globally generated, and its space of sections has the monomial basis
$\{\chi^m : m \in \Delta^\circ \cap \widetilde M\}$ \cite{Batyrev94,CLS}, where
$\widetilde M = \widetilde N^\vee$ is the dual lattice.  Throughout, $\Delta$ is the \emph{fan} polytope (Batyrev's
$\Delta^\ast$), so our $\Delta^\circ$ is his $\Delta$, the Newton polytope
of the anticanonical sections.

For a face $G \preceq \Delta$ write $\sigma_G$ for the cone over $G$,
$V(\sigma_G) \subset \PP_\Delta$ for the corresponding orbit closure, and
$G^\circ \preceq \Delta^\circ$ for the dual face.  If $F$ is a
two-dimensional face, then $\sigma_F$ is a three-dimensional cone,
$V(\sigma_F)$ is a toric curve $\cong \PP^1$, and $F^\circ$ is an edge of
$\Delta^\circ$; we write $\ell(F^\circ)$ for its lattice length.

\emph{Normalization convention.}  When listing example data we describe each
facet of $\Delta$ by the primitive integral functional $u$ normalized so that
the facet lies on $\{\langle u,\cdot\rangle = 1\}$ (we say $u$ is \emph{at
level one}); the corresponding vertex
of $\Delta^\circ$ is then $-u$, and the dual face $G^\circ$ of a face $G$ is
spanned by the \emph{negatives} of the facet functionals through $G$.  Since
lattice length is invariant under $u \mapsto -u$, we compute $\ell(F^\circ)$
directly on the listed functionals.\footnote{This matches the normalization
used by the ancillary programs.}  The
vertices $-u$ of $\Delta^\circ$ dual to the two facets of $\Delta$
containing $F$ restrict
to the sublattice $N_F = \widetilde N \cap \RR\sigma_F$ as a single integral
height function equal to $-1$ on $F$; its negative (the common
restriction of the two listed functionals) is a height function equal to
$+1$ on $F$, so $(\sigma_F, N_F)$ is the Gorenstein cone
over the lattice polygon $F$ in the height-one convention of
\S\ref{sec:local}, and we write $C(F)$ for the associated toric
threefold, computed in the lattice induced on the
affine span of $F$.

\begin{proposition}\label{prop:planting}
Let $\Delta \subset \ZZ^4$ be reflexive, all of whose edges have lattice
length one, and let $X \in |{-K_{\PP_\Delta}}|$ be generic.  Then:
\begin{enumerate}
\item $X$ is a Calabi--Yau threefold with Gorenstein canonical singularities.
\item For each two-dimensional face $F \preceq \Delta$ with $C(F)$ singular,
$X$ meets the curve $V(\sigma_F)$ in exactly $\ell(F^\circ)$ points, all in
the open orbit, and at each such point $p$ the germ $(X,p)$ is analytically
isomorphic to $(C(F), 0)$.
\item $X$ is smooth away from these points.
\end{enumerate}
\end{proposition}

\begin{proof}
(1) is Batyrev's theorem \cite[\S4]{Batyrev94}.

(2) Fix such an $F$.  The restriction map on anticanonical sections sends
$\chi^m$ to a non-zero section on $V(\sigma_F)$ exactly when
$m \in F^\circ \cap \widetilde M$: indeed $\chi^m$ vanishes identically on
$V(\sigma_F)$ unless $\langle m, \cdot \rangle$ attains its minimum $-1$ on
$F$.  The restricted line bundle is the toric line bundle on
$V(\sigma_F)\cong\PP^1$ with one-dimensional moment polytope $F^\circ$, of
degree $\ell(F^\circ)$ \cite[\S6.3]{CLS}, and the lattice points of $F^\circ$
restrict to the full monomial basis of $H^0(\PP^1,
\mathcal{O}(\ell(F^\circ)))$.  Hence for generic $X$ the intersection
$X \cap V(\sigma_F)$ consists of $\ell(F^\circ)$ distinct reduced points
avoiding the two torus-fixed points.

For the local structure, split $\widetilde N \cong N_F \oplus \ZZ$; then the
open toric chart of $\sigma_F$ is $U_{\sigma_F} \cong C(F) \times \CC^\ast$,
with $V(\sigma_F)\cap U_{\sigma_F} = \{0\} \times \CC^\ast$
\cite[Prop.~3.3.9]{CLS}.  Trivialize $-K$ on this chart and embed
$C(F) \subset \CC^r$ by a generating set of $\sigma_F^\vee \cap M_F$
(writing $M_F = N_F^\vee$); the
defining equation of $X$ extends to a function $\widetilde f(z, t)$ on
$\CC^r\times\CC^\ast$.  At a point $p=(0,t_0)$ of the transverse intersection
above, $\widetilde f(0, \cdot)$ has a simple zero at $t_0$, so
$\partial\widetilde f/\partial t (p) \neq 0$, and by the implicit function
theorem $\{\widetilde f = 0\}$ is, near $p$, the graph of an analytic function
$t = \varphi(z)$.  Intersecting with $C(F)\times\CC^\ast$ and projecting along
the graph yields an analytic isomorphism $(X, p) \cong (C(F), 0)$.

(3) $\Sing \PP_\Delta$ is the union of the orbit closures of the non-smooth
cones of the face fan.  Cones over vertices are smooth: every vertex
of a reflexive polytope is primitive, since a lattice point of the dual
facet of $\Delta^\circ$ would otherwise pair with it outside $\ZZ$.
The full-dimensional cones'
smooth loci are irrelevant for a generic hypersurface: since
$-K_{\PP_\Delta}$ is globally generated, Bertini gives that $X$ is smooth away
from $\Sing\PP_\Delta$, and $X$ avoids the finitely many torus-fixed points
(the section $\chi^{u_\Gamma}$, for $u_\Gamma$ the vertex of $\Delta^\circ$
dual to a facet $\Gamma$, is non-vanishing at the fixed point of
$\sigma_\Gamma$).  For a two-dimensional cone $\sigma_e$ over an edge $e$ of
$\Delta$ with vertices $v, w$: since $\ell(e)=1$ and $v$ lies at height $1$
with respect to the (restriction of the) supporting functional of any facet
containing $e$, the pair $(v, w-v)$ is a basis of $\widetilde N \cap
\RR\sigma_e$, so $\sigma_e$ is smooth.  Hence, under the all-unit-edges
hypothesis, $\Sing\PP_\Delta$ meets $X$ only in the curves $V(\sigma_F)$ of
part (2), and there only at the listed points.
\end{proof}

\begin{remark}\label{rem:longedges}
If $\Delta$ has an edge $e$ of lattice length $m \ge 2$, the same analysis
shows that the curve $X\cap V(\sigma_e)$ has generic transverse
singularity type $A_{m-1}$; here $V(\sigma_e)$ is a surface.  We impose the
all-unit-edges hypothesis only to keep the
singular locus zero-dimensional.  The local statement (2) at a
\emph{unit-edge} face $F$ is unaffected by longer edges elsewhere: the
points of $X$ on $V(\sigma_F)$ remain isolated singular points with germ
$C(F)$, and part (1) holds for any reflexive $\Delta$, so
Corollary~\ref{cor:global} below applies to every reflexive $\Delta$
possessing a unit-edge face of type \textup{(R)} or \textup{(D)}, whether
or not the remaining edges are unit.
\end{remark}

\section{Local-to-global: no smoothing}\label{sec:nosmoothing}

\begin{lemma}\label{lem:localglobal}
Let $X$ be a reduced projective variety with an isolated singular point $p$
such that no irreducible component of the reduced miniversal base space of the
germ $(X,p)$ has smooth general fibre.  Then for every flat proper family
$\mathcal{X} \to (T,0)$ over a reduced pointed base with
$\mathcal{X}_0 \cong X$, every fibre $\mathcal{X}_t$ for $t$ near $0$ is
singular.  In particular $X$ admits no smoothing.

If moreover the reduced miniversal base of $(X,p)$ is a point, then the
induced family of germs at $p$ is trivial: every $\mathcal{X}_t$ has a
singular point with germ isomorphic to $(X,p)$.
\end{lemma}

\begin{proof}
By Grauert \cite{Grauert72}, the isolated singularity $(X,p)$ has a miniversal
deformation $\mathcal{S} \to (S,0)$, which we represent on a suitable Milnor
representative; versality holds for analytic families over reduced base germs
after shrinking \cite[II.1]{GLS}.  The relative singular locus
$\Sigma \subset \mathcal{S}$ is closed analytic, and (after shrinking) the
projection $\Sigma \to S$ is finite: its central fibre is the
isolated point $p$, so on a Milnor representative, whose boundary the
nearby singular loci avoid, the projection is proper with finite
fibres.  Its image is therefore a closed analytic subset
$D \subseteq S$, the locus of singular fibres.  A component of
$S_{\mathrm{red}}$ is a smoothing component precisely when its underlying set is not
contained in $|D|$; by hypothesis there are none, so $|D|=|S|$ and
\emph{every} fibre of the miniversal family is singular.

Now let $\mathcal{X} \to (T,0)$ be as in the statement and restrict it to a
Milnor representative around $p$; after shrinking $T$ this restriction is
induced from $\mathcal{S} \to S$ by a base-change map $h : (T,0) \to (S,0)$.
Its fibres are fibres of the miniversal family, hence singular; so
$\mathcal{X}_t$ has a singular point near $p$ for every $t$ near $0$.

For the final statement: if $S_{\mathrm{red}}$ is a point, then $h$ factors
through $S_{\mathrm{red}}$ because $T$ is reduced, i.e.\ $h$ is constant, and
the induced family of germs is the pullback of the central fibre.
\end{proof}

\begin{remark}
Lemma~\ref{lem:localglobal} requires no Calabi--Yau hypothesis and no
obstruction theory, because the local model already forbids smooth nearby
fibres.  For the much finer relationship between local and global
deformations of a Calabi--Yau threefold with isolated canonical
singularities, in particular when local smoothings do exist and the
question is whether they are realized by global deformations, see Gross
\cite[\S\S1--2]{Gross97} and Namikawa's stratified local moduli
\cite{Namikawa02}.
\end{remark}

\begin{corollary}\label{cor:global}
Let $\Delta$ be a reflexive $4$-polytope possessing a two-dimensional face
$F$ whose edges have lattice length one and such that $F$ is of type
\textup{(R)} or \textup{(D)} in Theorem~\textup{\ref{thm:trichotomy}}.  Then the generic
anticanonical hypersurface $X \subset \PP_\Delta$ is a Calabi--Yau threefold
with canonical Gorenstein singularities which admits no smoothing.
\end{corollary}

\begin{proof}
By Proposition~\ref{prop:planting} and Remark~\ref{rem:longedges}, $X$ has
$\ell(F^\circ) \ge 1$ points with
germ $C(F)$; by Theorem~\ref{thm:trichotomy} the miniversal base of this germ
has no smoothing component; by Lemma~\ref{lem:localglobal} no deformation of
$X$ over a reduced base (in particular no smoothing over a disc)
has smooth general fibre.
\end{proof}

\section{The examples}\label{sec:examples}

We first construct the three-point examples used to explain the method,
then the single-point examples stated in the introduction.  Write
$e_1,\dots,e_4$ for the standard basis of $\ZZ^4$.
The first two examples realize a non-smoothable polygon $Q$ from \S\ref{sec:local} as a
two-dimensional face of a reflexive $4$-polytope of the simple shape
$\conv\left(Q \times \{(1,0)\} \cup \{e_4,\ w\}\right)$, found by a small
computer search over the completing vertex $w$.  (Placing $Q$ in the plane
$x_3=1,\,x_4=0$ embeds it in a saturated affine sublattice, so the face
polygon carries the intrinsic lattice structure of $Q$; translations of $Q$
within that plane and the normalization of the first completing vertex to
$e_4$ are unimodular coordinate changes.)

\subsection{A Calabi--Yau threefold with three \texorpdfstring{$\FF_1$}{F1}-cone points}

\begin{theorem}\label{thm:A}
Let $Q_A$ be the reduced-rigid quadrilateral with
$E(Q_A)=\{(-2,-1),\allowbreak(0,-1),\allowbreak(1,0),\allowbreak(1,2)\}$, embedded with vertices
$(0,0),(-2,-1),(-2,-2),(-1,-2)$, and let
\[
\Delta_A=\conv\bigl(Q_A\times\{(1,0)\}\ \cup\ \{e_4,\ (1,1,-1,-1)\}\bigr)
\subset \ZZ^4 .
\]
Then:
\begin{enumerate}
\item $\Delta_A$ is reflexive with $6$ vertices and $6$ facets, and all $13$
edges of $\Delta_A$ have lattice length one.  Its $13$ two-dimensional faces
consist of the face $F_A = Q_A\times\{(1,0)\}$ and $12$ standard triangles.
\item The generic anticanonical hypersurface $X_A \subset \PP_{\Delta_A}$ is a
Calabi--Yau threefold whose singular locus consists of exactly
$\ell(F_A^\circ) = 3$ points, at each of which the germ of $X_A$ is the
anticanonical cone over $\FF_1$.
\item Every deformation of $X_A$ over a reduced base preserves the
three singular germs up to analytic isomorphism.  In particular $X_A$ admits
no smoothing.
\item Batyrev's formulas give $(h^{1,1},h^{2,1}) = (5,101)$, hence Euler
number $-192$, for the maximal projective crepant partial resolution
$\widehat X_A$.
\end{enumerate}
\end{theorem}

\begin{proof}
(1) is a finite verification.  The six facet inequalities
$\langle u, \cdot\rangle \le 1$ of $\Delta_A$ have primitive normals
\[
(0,0,1,1),\ (0,0,1,-2),\ (-3,6,1,1),\ (6,-3,1,1),\ (-3,0,-5,1),\
(0,-3,-5,1),
\]
each at level exactly $1$, which is reflexivity; the face $F_A$ is the
intersection of $\Delta_A$ with the two supporting hyperplanes of the first
two normals.  The remaining two-dimensional faces admit a uniform
description: writing $w=(1,1,-1,-1)$ for the second completing vertex, the
segment $[e_4,w]$ is an edge of $\Delta_A$, and the twelve triangles are the
joins of the four edges of $F_A$ with $e_4$, the joins of those edges with
$w$, and the joins of the four vertices of $F_A$ with the edge $[e_4,w]$;
each is unimodular, and every edge of $\Delta_A$ has lattice length one
(all verified in exact arithmetic).  The dual edge $F_A^\circ$ is spanned by the negatives
of the first two normals (the convention of \S\ref{sec:planting}), of lattice
length $\gcd\bigl((0,0,1,1)-(0,0,1,-2)\bigr)=3$.

(2) By Proposition~\ref{prop:planting}, the singular points of $X_A$ are: for
$F_A$, three points with germ $C(Q_A)$; for the twelve standard-triangle
faces, none, as their cones are smooth.  By Remark~\ref{rem:delpezzo},
$C(Q_A)$ is the anticanonical cone over $\FF_1$.

(3) $Q_A$ is reduced-rigid: $E(Q_A)$ has no proper zero-sum sub-multiset (a direct
check of the $14$ proper non-empty subsets).  By
Theorem~\ref{thm:trichotomy}(R) the reduced miniversal base of each germ is a
point, and Lemma~\ref{lem:localglobal} applies, including its final statement.

(4) With $|\Delta_A\cap\ZZ^4|=8$ and $|\Delta_A^\circ\cap\ZZ^4|=130$ lattice
points and
the face data above, Batyrev's formulas \cite{Batyrev94} for the Hodge numbers
of $\widehat X_A$ give $h^{1,1}=5$ and $h^{2,1}=101$.\footnote{As
validation, the same computation gives $(1,101)$ for the quintic simplex
and $(4,112)$ for $X_9\subset\PP(1,1,1,3,3)$.}
\end{proof}

\begin{remark}
The three germs of $X_A$ have $\dim T^1 = 1$ but fat-point miniversal bases:
each germ has first-order deformations, all obstructed at higher order.
This is exactly the local model excluded in Gross's smoothability theorem for
primitive type~II contractions \cite[Theorem~5.8, \S5.5]{Gross97}, here
realized three times on an explicit compact Calabi--Yau threefold in the
Kreuzer--Skarke list.
\end{remark}

\subsection{Three locally deformable non-smoothable singular points}

\begin{theorem}\label{thm:B}
Let $Q_B = T + s$ be the type-\textup{(D)} pentagon of
Example~\textup{\ref{ex:pentagon}},
embedded with vertices $(0,0),(-2,-1),(-3,-2),(-2,-3),(-1,-2)$, and let
\[
\Delta_B=\conv\bigl(Q_B\times\{(1,0)\}\ \cup\ \{e_4,\ (1,1,-1,-1)\}\bigr)
\subset \ZZ^4 .
\]
Then:
\begin{enumerate}
\item $\Delta_B$ is reflexive with $7$ vertices and $7$ facets; all $16$ edges
have lattice length one; its $16$ two-dimensional faces are
$F_B = Q_B \times \{(1,0)\}$ and $15$ standard triangles.
\item The generic anticanonical $X_B \subset \PP_{\Delta_B}$ is a Calabi--Yau
threefold whose singular locus consists of exactly $\ell(F_B^\circ)=3$ points,
each with germ $C(Q_B)$.
\item Each singular germ of $X_B$ admits a non-trivial one-parameter
deformation: its reduced miniversal base is a smooth curve.  Nevertheless the
miniversal base has no smoothing component, and $X_B$ admits no smoothing.
\item $(h^{1,1},h^{2,1}) = (9,81)$ for the maximal projective crepant partial
resolution, hence Euler number $-144$.
\end{enumerate}
\end{theorem}

\begin{proof}
(1) The seven facet normals are
\[
\begin{gathered}
(0,0,1,1),\ (0,0,1,-2),\ (-3,6,1,1),\ (6,-3,1,1),\\
(-3,3,-2,1),\ (-1,-1,-4,1),\ (3,-3,-2,1),
\end{gathered}
\]
each at level $1$; the rest is finite verification as in
Theorem~\ref{thm:A}(1), with $F_B^\circ$ again of length $3$; the fifteen
triangles are again the joins of the five edges of $F_B$ with each
completing vertex and of the five vertices of $F_B$ with the edge joining
the completing vertices.
(2) is Proposition~\ref{prop:planting} plus the face inventory.
(3) is Example~\ref{ex:pentagon} combined with
Theorem~\ref{thm:trichotomy} and Lemma~\ref{lem:localglobal}.
(4) is computed as in Theorem~\ref{thm:A}(4), with
$|\Delta_B\cap\ZZ^4|=10$, $|\Delta_B^\circ\cap\ZZ^4|=104$.
\end{proof}

\begin{remark}\label{rem:stuck}
By Lemma~\ref{lem:cascade}, the available local deformation at each singular
point of $X_B$ replaces the pentagon cone by a $\frac13(1,1,1)$ point.  The
global deformation construction of \cite[Proposition~4.7]{paper2} applies here:
it gives a deformation of $X_B$, induced by an explicit trinomial deformation
of the ambient pair, whose general neighbouring member has three such rigid
quotient points.  Thus these local deformations are globally realized,
and their residual quotient germs admit no further deformation.
This differs from the global obstruction in \cite[Example~4.1]{Gross97},
where locally smoothable germs cannot be smoothed by a deformation of the
compact threefold: here each initial germ itself has no smoothing component.
\end{remark}

\subsection{A single \texorpdfstring{$\FF_1$}{F1}-cone point}

The examples above have three singular points because the dual edge of the
prescribed face has lattice length $3$.  Within the two-vertex family this
cannot be pushed below $2$.

\begin{lemma}[two adjoined vertices force $\ell(F^\circ)\ge2$]
\label{lem:twovertices}
Let $Q$ be a unit-edge lattice polygon, let
$\Delta = \conv\bigl(Q\times\{(1,0)\}\cup\{e_4, w\}\bigr) \subset \ZZ^4$ be
reflexive, and let $F = Q\times\{(1,0)\}$.  Then $F$ is a
two-dimensional face of $\Delta$ and $\ell(F^\circ) \ge 2$.
\end{lemma}

\begin{proof}
$F$ is a face: reflexivity puts $0$ in the interior, which (as shown
below) forces $w_3 \le -1$, so $u_0 = (0,0,1,0)$ takes the value $1$ on
$Q\times\{(1,0)\}$ and less on $e_4$ and $w$, i.e.\ it supports
$\Delta$ exactly along $F$.  Every facet of $\Delta$ containing
$F$ has normal of the form $u_t = (0,0,1,t)$ with
$t \in \ZZ$ (these are exactly the integral functionals restricting to $1$
on the affine span of $F$), and the two facets through $F$ are the extreme
supporting members of this pencil: for
$t_{\min} < t < t_{\max}$ the face of $u_t$ is
$\operatorname{face}(u_{t_{\min}}) \cap \operatorname{face}(u_{t_{\max}})
= F$, so the interval $[t_{\min}, t_{\max}]$ is exactly $F^\circ$ and
$\ell(F^\circ) = t_{\max} - t_{\min}$.  With completing vertices $e_4$ and
$w = (w_1,w_2,w_3,w_4)$ the bound $\ell(F^\circ)\ge 2$ is forced, as
follows.  All vertices of $Q\times\{(1,0)\}$ have $x_3 = 1$ and $x_4 = 0$,
and $e_4$ has $x_3 = 0$, $x_4=1$; for the origin to be interior some
vertex must have negative third coordinate and some vertex negative
fourth coordinate, so
$w_3 \le -1$ and $w_4 \le -1$.  The member $u_t$ supports $\Delta$ if and
only if $\langle u_t, e_4\rangle = t \le 1$ and
$\langle u_t, w\rangle = w_3 + t w_4 \le 1$; since $w_4 < 0$ the second
constraint bounds $t$ from below only, so $t_{\max} = 1$, attained on the
facet through $e_4$.  The facet at the lower extreme must contain a vertex
outside the plane of $F$, necessarily $w$, so
$t_{\min} = (1-w_3)/w_4$, an integer with positive numerator ($\ge 2$) and
negative denominator; hence $t_{\min} \le -1$ and
$\ell(F^\circ) = 1 - t_{\min} \ge 2$.
\end{proof}

A third adjoined vertex in $\{x_3 = 1\}$ removes the interiority constraint
that forced $w_3 \le -1$, and realizes $t_{\min} = 0$:

\begin{theorem}\label{thm:C}
Let $Q_A$ be embedded as in Theorem~\textup{\ref{thm:A}} and let
\[
\Delta_C=\conv\bigl(Q_A\times\{(1,0)\}\ \cup\ \{e_4,\ (1,1,-1,0),\
(-2,-2,1,-1)\}\bigr)\subset\ZZ^4 .
\]
Then:
\begin{enumerate}
\item $\Delta_C$ is reflexive with $7$ vertices and $10$ facets; all $18$
edges have lattice length one; its $21$ two-dimensional faces are
$F_C = Q_A\times\{(1,0)\}$ and $20$ standard triangles; and the dual edge
$F_C^\circ$, spanned by the negatives of the facet functionals
$(0,0,1,0)$ and $(0,0,1,1)$, has lattice length $1$.
\item The generic anticanonical hypersurface $X_C \subset \PP_{\Delta_C}$ is
a Calabi--Yau threefold whose singular locus is a \emph{single} point, at
which the germ is the anticanonical cone over $\FF_1$.  Every deformation of
$X_C$ over a reduced base preserves this germ; in particular
$X_C$ admits no smoothing.
\item The maximal projective crepant partial resolution $\widehat X_C$ has
$(h^{1,1},h^{2,1}) = (4,108)$, hence Euler number $-208$.
\end{enumerate}
\end{theorem}

\begin{proof}
The ten facet normals are
\[
\begin{array}{lllll}
(0,0,1,1), & (0,0,1,0), & (-2,0,-3,0), & (-2,0,-3,1), & (-2,4,1,-4),\\[2pt]
(-2,4,1,1), & (0,-2,-3,0), & (0,-2,-3,1), & (4,-2,1,-4), & (4,-2,1,1),
\end{array}
\]
each at level exactly $1$; the face inventory, edge lengths, and the dual
edge $F_C^\circ$
are finite verifications as in Theorem~\ref{thm:A}(1).  Parts (2) and (3)
follow exactly as in Theorem~\ref{thm:A}, with
$|\Delta_C\cap\ZZ^4| = 9$, $|\Delta_C^\circ\cap\ZZ^4| = 141$ in Batyrev's
formulas.
\end{proof}

\begin{remark}
The same pair of completing vertices applied to the $\frac13(1,1,1)$
triangle (embedded with vertices $(0,0),(-2,-1),(-1,-2)$) gives a $6$-vertex
reflexive polytope whose generic anticanonical hypersurface is a Calabi--Yau
threefold with a \emph{single} $\frac13(1,1,1)$ point, with resolution Hodge
numbers $(3,111)$ and $\chi = -216$; the classical
$X_9 \subset \PP(1,1,1,3,3)$ has three such points.
\end{remark}

\subsection{A single locally deformable non-smoothable point}

We now construct a threefold whose unique singularity is the deformable
non-smoothable pentagon cone.  Our completion searches
realize $\ell(F^\circ)=1$ for type-\textup{(D)} faces only on polytopes
carrying additional singular faces; the example below was instead found in
the Kreuzer--Skarke classification itself, by the scan described in
\S\ref{sec:questions}.

\begin{theorem}\label{thm:D}
Let
\begin{align*}
\Delta_D=\conv\{&(1,0,0,0),\,(0,1,0,0),\,(0,-1,0,0),\,(0,0,1,0),\,
(-3,1,-1,0),\\
&(0,1,1,0),\,(0,0,0,1),\,(-3,2,1,-1),\,(-2,2,2,-1),\,(1,0,1,1)\}
\subset\ZZ^4 .
\end{align*}
Then:
\begin{enumerate}
\item $\Delta_D$ is reflexive with $10$ vertices and $15$ facets; all $30$
edges have lattice length one; its $35$ two-dimensional faces are $34$
standard triangles and one pentagon $F_D$, which is
$GL_2(\ZZ)$-equivalent to the type-\textup{(D)} pentagon $Q_B = T + s$ of
Theorem~\textup{\ref{thm:B}}; and the dual edge $F_D^\circ$, spanned by the
negatives of the facet functionals $(0,1,0,1)$ and $(-1,-1,1,1)$, has
lattice length $1$.
\item The generic anticanonical hypersurface $X_D \subset \PP_{\Delta_D}$ is
a Calabi--Yau threefold whose singular locus is a \emph{single} point, with
germ $C(Q_B)$: the reduced miniversal base of the singularity is a smooth
curve, so the singular point of $X_D$ admits a non-trivial one-parameter
local deformation, and yet $X_D$ admits no smoothing.
\item The maximal projective crepant partial resolution $\widehat X_D$ has
$(h^{1,1},h^{2,1}) = (8,118)$, hence Euler number $-220$.
\end{enumerate}
\end{theorem}

\begin{proof}
All statements in (1) are finite verifications in exact arithmetic as in
Theorem~\ref{thm:A}(1): the fifteen facet normals are
\[
\begin{array}{lllll}
(0,1,0,1), & (-1,-1,1,1), & (1,1,0,0), & (1,1,0,-1), & (1,0,1,-1),\\[2pt]
(1,1,-1,1), & (1,1,-1,-3), & (1,1,-3,1), & (1,1,-3,-5), & (1,-1,1,-1),\\[2pt]
(1,-1,1,-3), & (1,-1,-1,1), & (1,-1,-1,-7), & (1,-1,-5,1), & (1,-1,-5,-11),
\end{array}
\]
each at level exactly $1$; the pentagon face $F_D$ has vertices
$5,7,8,9,10$ in the order listed, its edge multiset in the induced lattice
is $\{(1,0),(3,2),(-1,0),(-3,-1),(0,-1)\}$, and the unimodular map
$\left(\begin{smallmatrix}-1&2\\-1&1\end{smallmatrix}\right)$
carries it to $E(Q_B)$; the first two facet normals listed are those of the
two facets through $F_D$, and the dual edge they determine has lattice
length $1$.  Part (2) follows from
Proposition~\ref{prop:planting}, Example~\ref{ex:pentagon},
Theorem~\ref{thm:trichotomy} and Lemma~\ref{lem:localglobal} exactly as in
Theorem~\ref{thm:B}; part (3) is Batyrev's formulas with
$|\Delta_D\cap\ZZ^4| = 15$, $|\Delta_D^\circ\cap\ZZ^4| = 152$.
\end{proof}

\begin{remark}
Theorems~\ref{thm:C} and~\ref{thm:D} together answer the minimality
question for both non-smoothable types: one reduced-rigid point, or one deformable
point, on a compact Calabi--Yau threefold, each realized inside the
Kreuzer--Skarke classification.  The scan that produced $\Delta_D$ found
three $10$-vertex polytopes whose unique \emph{non-smoothable} face is of
type \textup{(D)} with dual edge of length one; $\Delta_D$ is the one whose
other faces are all smooth (the other two carry additional ordinary
smoothable singular points).  Subject to the database hypothesis stated in
\S\ref{sec:questions}, over the full classification there are exactly four
polytopes whose generic hypersurface has a \emph{single}
singular point of type \textup{(D)}, with $10$, $11$, $11$ and $12$ vertices
($\Delta_D$ being the minimal one), and in all four the germ is $C(Q_B)$;
see \S\ref{sec:questions}.
\end{remark}

\section{Occurrence in the Kreuzer--Skarke classification and further questions}\label{sec:questions}

\subsection{Completion searches and the database enumeration}
The completion search that produced $\Delta_A$ and $\Delta_B$ is the simplest
possible: complete the embedded polygon by two vertices, one of which can be
normalized to $e_4$.  This family realizes $18$ of the $87$ non-smoothable
census classes of Proposition~\ref{prop:census} as two-dimensional faces of
reflexive $4$-polytopes (both examples above use the same completing vertex
$(1,1,-1,-1)$), and the count is stable: enlarging the coordinate box for
the completing vertex from $[-3,3]$ to $[-5,5]$ produces no new classes, and
every class realized this way has at most $6$ interior lattice points.
Completing by \emph{three} vertices instead (as in
Theorem~\ref{thm:C}, with both free vertices ranging over the box
$[-2,2]^4$) raises the count to $32$ of the $87$: all $23$ classes with
$i \le 6$ embed, together with $9$ classes with $7 \le i \le 10$.

A complementary approach is to scan a
per-vertex-count copy of the Kreuzer--Skarke classification.  We carried out
the test of Theorem~\ref{thm:trichotomy} on that copy and on the remaining
$36$-vertex polytope separately \cite{KreuzerSkarke00}.  The resulting global
counts have the explicit database-completeness hypothesis in the next
statement; the checks supporting it are described at the end of the
subsection.

\begin{proposition}[Frequency of the local obstruction]\label{prop:scan}
Assume that the per-vertex-count database files described below contain,
without repetition, exactly the reflexive $4$-polytopes in the vertex range
specified there.  Then, of the $473{,}800{,}776$ reflexive $4$-polytopes,
exactly $39{,}175{,}536$
(that is, $8.27\%$) carry a two-dimensional face with primitive edges whose
cone is not smoothable.  For each of these the generic anticanonical
hypersurface is a Calabi--Yau threefold admitting no smoothing, by
Corollary~\textup{\ref{cor:global}} \textup{(}via
Remark~\textup{\ref{rem:longedges}} for those carrying longer edges
elsewhere\textup{)}.
\end{proposition}

\begin{proof}
By Theorem~\ref{thm:trichotomy} the type of $C(F)$ depends only on the edge
multiset of $F$, so the test is a finite computation on each polytope; it was
run on every entry of the database copy and on the separately reconstructed
$36$-vertex polytope.  Under the stated completeness and no-duplication
hypothesis, these are precisely the members of the classification, and the
surviving counts are recorded below.
\end{proof}

Under the same database hypothesis, in more detail we find:
\begin{itemize}
\item $8.27\%$ (that is, $39{,}175{,}536$ of them) carry a unit-edge
non-smoothable two-dimensional face, so their generic anticanonical
hypersurfaces are non-smoothable Calabi--Yau threefolds by
Corollary~\ref{cor:global} (via Remark~\ref{rem:longedges} for the
polytopes carrying longer edges elsewhere); the fraction rises from $11.7\%$ at $5$
vertices to a peak of $14.0\%$ at $9$ and then declines steadily with the
vertex count, to below $1\%$ beyond $22$ vertices;
\item faces of type \textup{(D)} appear from $7$ vertices on (as they
must: a unit-edge triangle is indecomposable and a decomposable unit-edge
quadrilateral splits into two segment pairs and is smoothable, so a
type-\textup{(D)} polygon has at least five edges, and a
pentagonal two-face together with the two further vertices needed to span
dimension four already requires seven vertices), $497{,}434$ face
occurrences in all;
\item combining with the completion searches above, exactly $65$ of the
$87$ census classes are realized as two-dimensional faces of reflexive
$4$-polytopes.  The $22$ classes that never occur are all of type
\textup{(D)} and all have $i \ge 11$ interior points, so the occurrence of
a class decays with its size and, beyond a threshold, ceases entirely,
an obstruction we do not explain;
\item unit-edge non-smoothable faces \emph{outside} the range of
Proposition~\ref{prop:census} also
occur (reduced-rigid polygons with an edge coordinate exceeding $2$; those we
inspected have at most five edges), so the classes of
Proposition~\ref{prop:census} do not exhaust the phenomenon.
\end{itemize}
Under that hypothesis, the scan determines the fraction of reflexive
$4$-polytopes carrying a non-smoothable unit-edge two-face.  It therefore
gives a lower bound for the proportion whose generic anticanonical
hypersurface is non-smoothable, rather than a complete count of globally
non-smoothable hypersurfaces.  The counts concern polytopes up to lattice
equivalence, not deformation-equivalence classes of threefolds.
The scan also determines which local models occur within
Proposition~\ref{prop:census}; the unit-edge non-smoothable faces outside
that range remain to be classified.  Under the same hypothesis, the
single-point examples are similarly constrained: among
all reflexive $4$-polytopes with unit edges whose generic hypersurface
has a \emph{single} singular point, and that point of type \textup{(D)},
the point is the pentagon cone $C(Q_B)$ of Theorem~\ref{thm:D} in
\emph{every} case (the four such polytopes on the list all give
$C(Q_B)$); no other type-\textup{(D)} germ occurs as the sole singularity
of a Batyrev threefold.  (Type-\textup{(D)} germs of other classes do
occur as the
unique \emph{non-smoothable} point of a hypersurface that also carries
smoothable singular points; it is only as a lone singularity that
$C(Q_B)$ is forced.)

\emph{Completeness of the enumeration.}  We scanned $473{,}800{,}775$
polytopes through the per-vertex-count copy of the Kreuzer--Skarke
database that we used, whose files cover $5$ to $33$ vertices, and the
one remaining polytope separately.  Granting that the copy's entries are
distinct members of the classification (the polytope counts and Hodge
data recorded for each vertex number were validated throughout the
scan), the one polytope it omits is the classification's unique member
with more than $33$ vertices, and we identify it: it is the $36$-vertex
product of two reflexive hexagons, verified
directly.\footnote{Ancillary file \texttt{missing\_polytope.py}.}  That
polytope's $48$ two-dimensional faces are $36$ unit parallelograms and
$12$ hexagons of $\mathrm{dP}_6$ type, all of type \textup{(S)} with dual
edge of length one, so it contributes to none of the counts above.
For reproducibility, the ancillary input manifest pins the database copy to
repository revision
\texttt{60c0e119a036} (with the full identifier in the manifest) and records
the byte size and SHA-256 digest of every per-vertex-count file.  This
identifies the finite input exactly; it does not replace the completeness
and no-duplication hypothesis in Proposition~\ref{prop:scan}.

\subsection{Non-isolated singularities}

\begin{question}
The faces with non-unit edges, whose germs are non-isolated and lie beyond
Theorem~\ref{thm:trichotomy}, deserve their own count.  Filip's
$0$-mutability criterion \cite{Filip25} detects when his distinguished formal
families have smooth general fibre, but a classification of all smoothing
components in this setting remains open.
\end{question}

\subsection{The smoothable faces and the local-to-global problem}
Corollary~\ref{cor:global} uses only the non-smoothable local models.  In the
opposite regime, where all two-dimensional faces are of type \textup{(S)},
local
smoothings exist at every singular point, and the question becomes whether
they can be realized simultaneously by a global deformation; for ordinary
double points this is Friedman's criterion \cite{Friedman86}, and for
Calabi--Yau threefolds with canonical singularities the local-to-global
obstruction theory of \cite{Gross97,Namikawa02} applies.  A combinatorial
criterion on $\Delta$ for global smoothability of the Batyrev hypersurface,
in the spirit of the Fano-threefold results of \cite{Petracci20,CHP24}, is
the natural continuation of this note; it is the subject of the companion
paper \cite{paper2}, which smooths the points over a type-\textup{(S)} face
whenever its dual edge has lattice length at least two, shows the threshold
sharp, and compares the resulting criterion with the census of
Batyrev--Kreuzer \cite{BK10} of the all-conifold regime.
For nodes and anticanonical cones over smooth del Pezzo surfaces of
degrees five, six and seven, or over $\PP^1\times\PP^1$, the later
paper~\cite{paper5} gives a necessary and sufficient local-to-global
smoothing criterion for these singularity types.

\subsection{Mirror symmetry}
Corti--Filip--Petracci conjecture that the smoothing components of
$\Def C(Q)$ correspond to the $0$-mutable Laurent polynomials with Newton
polygon $Q$ \cite{CFP}, a mirror-theoretic statement.  What the presence of
a type \textup{(R)} or \textup{(D)} face of $\Delta$, and hence the
non-smoothability of $X$, means for the mirror family of the Batyrev pair
$(\Delta, \Delta^\circ)$ is taken up in \cite{paper3}, which classifies the
pairs whose \emph{both} hypersurfaces have isolated singularities and finds
that a type-\textup{(D)} face never occurs among them.

\bigskip
\noindent{\sc Bernd Johannes Wuebben, New York, NY,} \texttt{wuebben@gmail.com}

\end{document}